\documentclass[11pt]{article}

\usepackage[unicode=true,pdfusetitle,
 bookmarks=true,bookmarksnumbered=false,bookmarksopen=false,
 breaklinks=false,pdfborder={0 0 0},pdfborderstyle={},backref=true,colorlinks=false]{hyperref}
\usepackage{amsmath, amssymb, amsthm, verbatim,enumerate,color}
\usepackage{indentfirst}
\numberwithin{equation}{section}

\title{Resilience for the Littlewood--Offord Problem}

\author{Afonso S. Bandeira \thanks{Department of Mathematics and Center for Data Science, Courant Institute
of Mathematical Sciences, NYU. Email: bandeira@cims.nyu.edu. ASB acknowledges support from NSF grant DMS-1317308, NSF grant DMS-1712730, and NSF grant DMS-1719545. Part of this work was done while ASB
was with the Department of Mathematics at the Massachusetts Institute
of Technology.} \and Asaf Ferber\thanks{Department of Applied Mathematics, MIT. Email: ferbera@mit.edu. Research is partially supported by an NSF grant 6935855.}\and
Matthew Kwan\thanks{Department of Mathematics, ETH Z\"urich. Email: matthew.kwan@math.ethz.ch.}}

\date{\today}
\allowdisplaybreaks

\theoremstyle{plain}
\newtheorem{theorem}{Theorem}[section]
\newtheorem{lemma}[theorem]{Lemma}
\newtheorem{claim}[theorem]{Claim}

\newtheorem{problem}[theorem]{Problem}
\theoremstyle{definition}

\newtheorem{example}[theorem]{Example}

\global\long\def\sign{\operatorname{sign}}
\global\long\def\floor#1{\left\lfloor #1\right\rfloor }
\global\long\def\ceil#1{\left\lceil #1\right\rceil }

\global\long\def\a{\boldsymbol{a}}
\global\long\def\x{\boldsymbol{\xi}}

\let\originalleft\left
\let\originalright\right
\renewcommand{\left}{\mathopen{}\mathclose\bgroup\originalleft}
\renewcommand{\right}{\aftergroup\egroup\originalright}

\begin{document}
\maketitle

\begin{abstract}
Consider the sum $X\left(\x\right)=\sum_{i=1}^n a_i\xi_i$,
where $\a=(a_i)_{i=1}^n$ is a sequence of non-zero reals and
$\x=(\xi_i)_{i=1}^n$ is a sequence of i.i.d.\ Rademacher random variables (that is, $\Pr[\xi_i=1]=\Pr[\xi_i=-1]=1/2$).
The classical Littlewood--Offord problem asks for the best possible upper bound on the concentration probabilities $\Pr[X=x]$.
In this paper we study a resilience version of the Littlewood--Offord problem: how many of the $\xi_i$ is an adversary typically allowed to change without being able to force concentration on a particular value? We solve this problem asymptotically, and present a few interesting open problems.
\end{abstract}

\section{Introduction}

Let $\a=\left(a_{i}\right)_{i=1}^{n}$
be a fixed sequence of nonzero real numbers, and for a sequence of i.i.d.\
(independent, identically distributed) Rademacher random variables $\x=\left(\xi_{i}\right)_{i=1}^{n}$ (meaning $\Pr[\xi_i=1]=\Pr[\xi_i=-1]=1/2$),
define the random sum
\[
X=X_{\a}\left(\x\right)=\sum_{i=1}^{n}a_{i}\xi_{i}.
\]

Sums of this form are ubiquitous in probability theory. For example, $X$ can be interpreted as the outcome of an unbiased random walk with step sizes given by $\a$. The central limit theorem asserts that if the $a_i$ are all equal then $X$ asymptotically has a normal distribution. More flexible variants of the central limit theorem allow the $a_{i}$ to differ to an extent, and give quantitative control over the distribution of $X$.
An important example is the Berry--Esseen theorem~\cite{Ber,Ess}, which gives an estimate for the probability that $X$ lies in a given interval, comparing it to the corresponding probability for an appropriately scaled normal distribution (we give a precise statement, adapted to our context, later in the paper). The Berry--Esseen theorem is effective when the $a_{i}$ are of the same order of magnitude, in which case it can be used to easily deduce the estimate
\[
\Pr\left[X=x\right]=O\left(\frac{1}{\sqrt{n}}\right)
\]
for any $x$. Qualitatively, it guarantees that $X$ is unlikely to be concentrated on any particular value ($X$ is \emph{anti-concentrated}).

Over half a century ago, in connection with their study of random polynomials, Littlewood
and Offord~\cite{LO} considered anti-concentration in the general setting
where
no assumption is made on $\a$, other than that its entries being nonzero.
The classical
result of Littlewood and Offord~\cite{LO} strengthened by Erd\H{o}s~\cite{ELO} states that no matter the choice of $\a\in (\mathbb R\setminus\{0\})^n$, for all $x\in \mathbb R$ we have
\[
\Pr\left[X=x\right]\leq \left. \binom{n}{\lfloor n/2\rfloor} \middle/ 2^{n}=O\left(\frac1{\sqrt{n}}\right),\right.
\]
which is sharp for the sequence $\a=(1,1,\dots,1)$. This result is particularly remarkable due to the fact that if one does not assume anything about the $a_i$, then the distribution of $X$ may be far from normal and Berry--Esseen type bounds may no longer be meaningful.

Erd\H os' proof of the above inequality was combinatorial and extremely simple, as follows. First, we can assume that each $a_i$ is positive, because changing the sign of some $a_i$ does not affect the distribution of $X$. Then, observe that a sign vector $\x\in\{-1,1\}^n$ can be identified with the subset $\{i:\xi_i=1\}$ of $\{1,\dots,n\}$, and under this identification each fiber $X^{-1}\left(x\right)=\left\{ \x:X\left(\x\right)=x\right\}$ corresponds to a \emph{Sperner family}\footnote{A Sperner family is a collection of subsets of $[n]$ in which no subset is included in any other. For more details on Sperner families, the reader is referred to the book of Bollob\'as~\cite{BolBook}.}. It then suffices to apply a classical bound for the maximal size of a Sperner family.

Since the Littlewood--Offord problem was first introduced, many variants of it have been addressed; one particularly interesting line of research involves the relationship between the
structure of $\a$ and the resulting concentration probability $\max_{x}\Pr\left[X=x\right]$.
Erd\H os and Moser~\cite{EM} and S\'ark\"ozy and Szemer\'edi
\cite{SS} considered the case where the $a_{i}$ are all distinct,
and showed that the stronger bound $\Pr\left[X=x\right]=O\left(n^{-3/2}\right)$
holds.
Hal\'asz~\cite{Hal} gave even stronger
bounds for sequences which are ``arithmetically unstructured''
in an appropriate sense. More recently, Tao and Vu~\cite{TV09b,TV10}
and Nguyen and Vu~\cite{NV11} investigated the inverse
problem of characterizing the arithmetic structure of $\a$ given
the concentration probability $\max_{x}\Pr\left[X=x\right]$.

Many fruitful connections have been found between Littlewood--Offord-type problems and various areas of mathematics. In particular,
Littlewood--Offord-type theorems are essential tools in some of the landmark results in random matrix theory (see for example~\cite{TV09a,TV09b}). In particular, the Littlewood--Offord theorem gives an upper bound on the probability that a particular row of a random $\pm1$ matrix is orthogonal to a given vector, and can thus be used (see for example~\cite[Section~14.2]{BolRandom}) to bound the probability that a Bernoulli random matrix is singular.

\subsection{Our Results}

In this paper we are interested in studying a ``resilience'' version of the Littlewood--Offord problem. Given a sequence $\a\in (\mathbb R\setminus \{0\})^n$ and a real number $x\in \mathbb R$, we know that most sequences $\x\in\{-1,1\}^n$ do not satisfy the event $\{X\left(\x\right)=x\}$. We are interested in understanding whether most sequences $\x$ are ``far'' from this event.
In order to make this question precise we need a few definitions. Given two sequences $\x,\x'\in\{-1,1\}^n$ we define $d\left(\x,\x'\right)$ to be the \emph{Hamming
distance} between $\x$ and $\x'$ (that is, $d\left(\x,\x'\right)$ denotes the number of coordinates in which
$\x$ and $\x'$ differ). If $S\subset\{-1,1\}^n$ is a subset of the hypercube we further define $d\left(\x,S\right)$ as the minimum Hamming distance from $\x$ to a point in $S$. Finally, for a fixed sequence $\a$ of non-zero reals and $\xi\in \{-1,1\}^n$, let us define
\[
R_{x}\left(\x\right):=R^{\a}_{x}\left(\x\right)=d\left(\x,X^{-1}\left(x\right)\right),
\]
which is the minimum number of signs one needs to change in $\x$ in order to satisfy $X=x$. (For completeness, if $X=x$ is impossible then we set $R_x(\x)=\infty$). We refer to $R_x(\x)$ as the \emph{resilience} of $\x$ with respect to the event $\{X\ne x\}$, and if $R_x> k$ we say $\x$ is \emph{$k$-resilient}.

Given $\a$ we define
\[
 q_{k}(\a)=\max_{x}\Pr\left[R^{\a}_{x}\left(\x\right)\le k\right].
\]
as the maximum probability that $\x$ fails to be $k$-resilient.
We also define $p_k(n)$ as the ``worst case'' for this probability over all sequences $\a\in (\mathbb{R}\setminus\{0\})^n$:
\[
p_{k}\left(n\right)=\max_{\a\in\{\mathbb{R}\setminus\{0\}\}^n}q_{k}(\a)
\]
Equivalently, $p_{k}(n)$ corresponds to the maximum volume of the $k$-neighbourhood of a suitable ``Boolean
hyperplane'' $X^{-1}\left(x\right)$ in the hypercube.

An immediate natural question is as follows:
\begin{problem}
\label{prob:individual}Given a non-negative integer $k$, what is the asymptotic behavior of $p_{k}\left(n\right)$ as $n\to\infty$?
\end{problem}

The Erd\H{o}s--Littlewood--Offord bound trivially gives $$p_{0}\left(n\right)=\Theta\left(1/\sqrt{n}\right).$$

Understanding the case $k=1$ already has non-trivial implications.
In fact, F\"uredi, Kahn and Kleitman~\cite{FKK90}
showed that there are Sperner families whose 1-neighbourhood comprises a
constant proportion of the hypercube, while we will see in Section~\ref{sec:individual-lower} that $p_{1}\left(n\right)\to0$. This demonstrates a special structural property of ``arithmetic''
Sperner families of the form $X^{-1}\left(x\right)$.

More generally, we
believe an especially interesting question is to understand the qualitative
behaviour of $p_{k}\left(n\right)$, as a function of $k$.
\begin{problem}
\label{prob:threshold}For which $k=k(n)$ does $p_{k}\left(n\right)\to0$ as $n\to\infty$?
\end{problem}
In other words, we are asking for which $k$ we can expect a typical
$\x\in\{-1,1\}^n$ to be $k$-resilient, regardless of the choice of $x$ and $\a$.
This question is especially compelling in view of the recent popularity
of resilience problems for random graphs (see for example the influential survey of Sudakov and Vu~\cite{SV08}),
and in view of questions asked by Vu~\cite[Conjectures~7.4-5]{Vu08} concerning the resilience of the singularity of random matrices. Specifically, Vu asked how many entries of a random $\pm1$ matrix one has to change (``globally'' or ``locally'') to make it singular; due to the connection between the Littlewood--Offord problem and singularity of random matrices, these conjectures were our initial motivation to investigate the questions treated in this paper.

Before stating our results, we compute the typical resilience for a few simple illustrative specific choices of $\a$ and $x$.

\begin{example}
\label{exa:111} Consider the case $\a=\left(1,\dots,1\right)$, and for simplicity assume $n$ is even. One can easily derive that for all even $x$ we have
$$\Pr[X=x]=\binom{n}{\frac{n+x}{2}}2^{-n}.$$ Standard binomial estimates show that with (say) 99\% certainty we have $|X|=\Theta(\sqrt{n})$. Noting that $R_0=|X|/2$, we typically have $R_0=\Theta(\sqrt{n})$.
\end{example}
\begin{example}
  Let us next consider the sequence $\a=(1,2,\ldots,n)$. Since all the $a_i$ are distinct, it follows from the result of S\'ark\"ozy and Szemer\'edi \cite{SS} that $q_0(\a)=O\left(n^{-3/2}\right)$. Moreover, changing $k$ signs of $\x$ can increase or decrease $X$ by no more than $kn$, so there are at most $2kn+1$ ways to affect $X$ by changing $k$ signs. Therefore, as long as $kn=o(n^{3/2})$ (that is, $k=o(n^{1/2})$), the union bound shows that for any $x$, typically $R_x\ge k$.
\end{example}

\begin{example}
\label{exa:doubling} Take $\a=\left(1,2,4,\dots,2^{n-1}\right)$. Note that $X$ can
take $2^{n}$ different values (the odd integers between $-2^{n}$
and $2^{n}$). This of course leads to the minimum possible concentration
probability $\max_x\Pr\left[X=x\right]=2^{-n}$. Each $x$ in the support
of $X$ can be obtained by exactly one $\x$, so $R_{x}$ has the
binomial distribution $\operatorname{Bin}\left(n,1/2\right)$ and is tightly concentrated around $n/2$ by a large deviation
inequality for the binomial distribution (see for example~\cite[Theorem~2.1]{JLR00}).
\end{example}

We can see from the above three examples that the type of additive structure influencing the concentration probability does contribute somewhat to the typical resilience. However, the following example shows that the typical resilience can be much more strongly influenced by small subsequences of $\a$.

\begin{example}
\label{exa:upper-bound}Let $k$ be the minimal integer such that
$k\ge\log_2 n$ and $n-k$ is odd. Define $\a$ by $a_{1}=\dots=a_{n-k}=1$
and $a_{n-k+i}=2^{i-1}$. For any $\x$, modifying at most $k$ coordinates
we can make $\sum_{i=1}^{k}\xi_{n-k+i}a_{n-k+i}$ equal to any odd
number between $-n$ and $n$, so in particular we can make it equal
to $-\sum_{i=1}^{n-k}\xi_{i}a_{i}$, so that $X=0$. This means $R_{0}=O\left(\log n\right)$
(with probability 1).
\end{example}

Somewhat surprisingly, there is a sequence which typically results in significantly lower resilience than Example~\ref{exa:upper-bound}.

\begin{theorem}
\label{thm:upper-bound}There exists a sequence $\a\in \left(\mathbb{R}\setminus \{0\}\right)^n$ such that for any fixed $\varepsilon>0$, a.a.s.\footnote{By ``asymptotically almost surely'', or ``a.a.s.'', we mean that
the probability of an event is $1-o\left(1\right)$. Here and for
the rest of the paper, asymptotics are on $n\to\infty$.}\ $R_{0}\left(\x\right)\le(1+\varepsilon)\log_3\log n$. (That is to say, for $k\ge (1+\varepsilon))\log_3\log n$, we have $p_k\to1$).
\end{theorem}

The crux of Example~\ref{exa:upper-bound} was the fact that one can form all non-negative integers less than $2^{k}$ with
sums of subsets of $\left\{ 1,2,4,\dots,2^{k-1}\right\} $. In other words, $\left\{ 1,\dots,2^{k-1}\right\} $ is an \emph{additive basis
}of $\left\{ 0,1,2,\dots,2^{k}-1\right\} $. The proof of Theorem~\ref{thm:upper-bound}, which we defer to Section~\ref{sec:upper-bound}, involves a more efficient additive basis construction, using an idea from a 1937 paper of Rohrbach~\cite{Roh37}.

We are also able to prove that Theorem~\ref{thm:upper-bound} is in fact optimal, essentially answering Problem~\ref{prob:threshold}.
\begin{theorem}
\label{thm:lower-bound}For any fixed $\varepsilon>0$, any $\a\in \left(\mathbb{R}\setminus \{0\}\right)^n$ and any $x\in \mathbb{R}$, a.a.s.\ $R_{x}\left(\x\right)\ge(1-\varepsilon)\log_3\log n$. (That is to say, for $k=(1-\varepsilon)\log_3\log n$, we have $p_k\to0$).
\end{theorem}

We prove Theorem~\ref{thm:lower-bound} in Section~\ref{sec:lower}.

As for Problem~\ref{prob:individual}, for each fixed $k$ we are able to find the asymptotics
of $p_{k}\left(n\right)$ up to a polylogarithmic factor, as stated in the next theorem.

\begin{theorem}
\label{thm:pk-individual}We have
\[
p_{1}=\Theta\left(n^{-1/6}\right),
\]
and for any fixed $k\geq 2$,
\[
p_{k}\left(n\right)=n^{-1/(2\times 3^k)}\log^{O\left(1\right)}n.
\]
\end{theorem}

\subsection{Notation}
For a set of indices $I\subseteq\left[n\right]$ define
$$X_{I}\left(\x\right)=\sum_{i\in I}a_{i}\xi_{i}$$
to be the ``part'' of $X$ corresponding to $I$.

We use standard asymptotic notation throughout. For functions $f=f(n)$ and $g=g(n)$ we write
$f=O(g)$ to mean there is a constant $C$ such that $|f|\le C|g|$, 
we write $f=\Omega(g)$ to mean there is a constant $c>0$ such that $f\ge c|g|$, 
we write $f=\Theta(g)$ to mean that $f=O(g)$ and $f=\Omega(g)$, 
and we write $f=o(g)$ or $g=\omega(f)$ to mean that $f/g\to 0$. 
All asymptotics are taken as $n\to \infty$. Also, for a real number $x$, the floor and ceiling functions are denoted $\floor{x}=\max\{i\in\mathbb{Z}:i\le x\}$ and $\ceil{x}=\min\{i\in\mathbb{Z}:i\ge x\}$. For a positive integer $i$, we write $[i]$ for the set $\{1,2,\dots,i\}$. Finally, all logarithms are base 2, unless specified otherwise.

\subsection{Structure of the paper}
The structure of the paper is as follows. In Section~\ref{sec:lower} we give a lower bound on typical resilience (proving Theorem~\ref{thm:lower-bound}), in Section~\ref{sec:upper-bound} we construct a sequence with low resilience (proving Theorem~\ref{thm:upper-bound}), and in Section~\ref{sec:p1} we estimate the asymptotics of $p_k(n)$ (proving Theorem~\ref{thm:pk-individual}).


\section{Lower bound for typical resilience}\label{sec:lower}

In this section we prove Theorem~\ref{thm:lower-bound}. The heart of the proof is the following recurrence relation for $p_k(n)$.

\begin{lemma}
\label{thm:recursive}Let $k\in\mathbb{N}$ and let $f:=f(n)\to \infty$ be any function satisfying $(k+1)f^2\log n<n$. Then, for some constant $C$,
\[
p_{k}(n)\le\sum_{\ell=1}^{k}\left(4(k+1)f^2\log n\right)^{\ell}\max_{n'}p_{k-\ell}\left(n'\right)+C\left(k/f+1/n\right),
\]
where the maximum is over all $n'$ satisfying $0\le n-n'\le 4(k+1)f^2\log n$.
\end{lemma}

We remark that Lemma~\ref{thm:recursive} is also used in the proof of Theorem~\ref{thm:pk-individual}.

\subsection{Proof of Lemma~\ref{thm:recursive}}

Before giving the details of the proof of Lemma~\ref{thm:recursive}, we give a brief outline of the ideas. Intuitively, we expect $X$ to typically have order of magnitude about its standard deviation (which is $\sqrt{\sum_i a_i^2}$). If this is much larger than any individual $a_i$ then we expect the resilience $R_0$ to be large, as flipping a sign in $\x$ has a relatively small impact on $X$. Therefore (as already suggested by Example~\ref{exa:upper-bound}), it is important to distinguish those $a_i$ that are ``abnormally large", and consider them separately.

So, the proof of Lemma~\ref{thm:recursive} starts by
isolating ``large'' $a_i$ such that $a_i^2$ is almost as large as the sum of the squares of all $a_j\le a_i$ (here ``almost as large'' is parameterized by the function $f$). If there are many such $a_i$, then for similar reasons as in Example~\ref{exa:doubling} the resilience is very likely to be high. We can therefore assume that there are a small quantity of such $a_i$; we need to give an upper bound on the probability of being able to make $X=x$ with up to $k$ sign changes.

First consider the case where $\ell\geq 1$ of the $k$ changes are made on the
``large'' numbers. Because there are few such numbers, it is not too wasteful to take the union bound over each possible way to make these changes. Then, we can recursively bound the probability that we can make $X=x$ with at most $k-\ell$ further changes to the ``small" numbers.

Otherwise, if none of the sign changes are made on ``large'' numbers, then as we have already explained, the typical size of $X$ is larger than one can ``cancel out'' without making a large number of sign flips, so the resilience is high. We will rigorously establish this fact using the Berry--Esseen theorem, as follows (this version of the Berry--Esseen theorem immediately follows from the statement in~\cite{Ess}).

\begin{theorem}\label{thm:BE} For $X=\sum_{i=1}^n a_i\xi_i$ as in the introduction, let $\sigma^{2}=\sum_{i=1}^{n}a_{i}^{2}$ be the variance of $X$, and let $\rho=\sum_{i=1}^{n}\left|a_{i}\right|^{3}$. Let $\Phi$ be the cumulative distribution function of the standard
normal distribution. Then,
\[
\left|\Pr\left[\frac{X}{\sigma}\le x\right]-\Phi\left(x\right)\right|=O\left(\frac{\rho}{\sigma^{3}}\right).
\]
\end{theorem}

Now we give the details of the proof of Lemma~\ref{thm:recursive}.
\begin{proof}[Proof of Lemma~\ref{thm:recursive}]
Fix $k>0$ and $\a$. Note that we may assume that all the $a_i$ are non-negative, as changing signs of any subset of the $a_i$ does not change the distribution of $X$. Moreover, by relabeling if necessary, we can assume that $$0\leq a_{1}\le\dots\le a_{n}.$$
We denote partial sums of squares as follows:
$$\sigma^2_{i}=\sum_{j=1}^{i}a_{j}^{2}.$$
Now, let $i_1:=n>i_2>\ldots >i_t$ be a longest subsequence of indices for which the following properties hold for all $j<t$:
\begin{enumerate}
\item $a_{i_j}\geq 2a_{i_{j+1}}$, and
\item for all $i>i_{j+1}$ we have $a_{i_{j}}<2a_i$.
\end{enumerate}

Note that Property~1 forces all possible signed sums of the $a_{i_j}$ to be distinct (that is, $X_{\{i_1,\dots,i_t\}}$ takes $2^t$ different values). Maximality and Property~2 imply that $a_i> a_{i_t}/2$ for all $i\in [n]$.

If $t$ is large, then the atom probabilities are small, and therefore the resilience is high. We summarize this in the following claim.
\begin{claim}
\label{claim:small-atom}
If $t> (k+1)\log n$ then $q_{k}(\a)<1/n$.
\end{claim}
\begin{proof}
Let $I=\left\{ i_{j}:j\le t\right\}$ and condition on the outcomes of the $\xi_j$, $j\notin I$. The random variable $X$ can then take $2^{t}$ different
values, each occurring with probability $2^{-t}$. This means that, unconditionally, the probability that $X$ is equal to any particular value is at most $2^{-t}$. Now, there are at most $n^k$ ways to change up to $k$ of the $\xi_i$, and given a particular choice of indices at which to perform changes, the resulting sequence $\x'$ has the same distribution as $\x$. Therefore, the probability that $X(\x')$ is equal to any particular value after this change is still at most $2^{-t}$, so by the union bound $q_{k}(\a)\le n^{k}2^{-t}< 1/n$ as desired.
\end{proof}

From now on we assume that $t\leq (k+1)\log n$. Let $\tau$
be the first $j$ for which $a_{i_{j}}\le\sigma_{i_{j}}/f$. (If there is no such $j$, we set $\tau=\infty$). This condition defining $\tau$ is chosen so that we will later be able to control $X_{[i_\tau]}$ via the Berry--Esseen theorem.
In the following claim we show that $\tau<\infty$ and moreover that $[i_\tau]$ comprises most of $[n]$.

\begin{claim}
We have $\tau\leq t$ and $n-i_{\tau}\leq 4(k+1)f^{2}\log n$.
\end{claim}

\begin{proof}
Note that for any $j$ with $a_{i_{j}}>\sigma_{i_{j}}/f$ we have $$\left|\left\{ i:\frac{a_{i_{j}}}2<a_{i}\le a_{i_{j}}\right\} \right|\le4f^{2}.$$
Indeed, otherwise we would have the contradiction $$\sigma_{i_{j}}^{2}>4f^{2}\left(\frac{a_{i_{j}}}2\right)^2=(f\,a_{i_j})^2>\sigma_{i_j}^2.$$

If we were to have $\tau=\infty$ this would mean $a_{i_{j}}>\sigma_{i_{j}}/f$
for all $j\le t$. Therefore, this would lead to the contradiction
\[
n=\left|\left\{ i:a_{i}>\frac{a_{i_{t}}}2\right\} \right|=\sum_{j=1}^{t}\left|\left\{ i:\frac {a_{i_{j}}}2<a_{i}\le a_{i_{j}}\right\} \right|\leq 4f^2t\leq 4f^{2}(k+1)\log n<n.
\]
Similarly, we have
\[
n-i_{\tau}=\left|\left\{ i:a_{i}>\frac{a_{i_{\tau-1}}}2\right\} \right|=\sum_{t=1}^{\tau-1}\left|\left\{ i:\frac{a_{i_{t}}}2<a_{i}\le a_{i_{t}}\right\} \right|\leq 4f^2t\leq 4(k+1)f^2\log n.\tag*{\qedhere}
\]
\end{proof}
Now, let $n'=i_\tau$, let $J=\left[n'\right]$ and let $I=\left[n\right]\backslash J$. For each $0\leq \ell\leq k$ we will consider the case where we change exactly $\ell$ elements of $\x|_I$, and we will then take a union bound over all $\ell$.

For $\ell>0$, there are at most $\left(4(k+1)f^2\log n\right)^{\ell}$
ways to modify $\ell$ elements of $\x|_{I}$. For each such possibility, we
can condition on the modified value of $\x|_{I}$ (therefore on $X_{I}\left(\x\right)$),
and for any $x$ the probability that we will be able to make $X_{J}=x-X_{I}$
with our remaining $k-\ell$ modifications is at most $p_{k-\ell}\left(n'\right)$
by induction. Therefore, the probability we can make $X=x$ while
modifying at least one element of $\x|_{I}$ is at most
\[
\sum_{\ell=1}^{k}\left(4f^2(k+1)\log n\right)^{\ell}p_{k-\ell}\left(n'\right).
\]
It remains to consider the possibility that we do not modify $\x|_{I}$
at all. Again, condition on $\x|_{I}$ (therefore on $X_{I}$).
Note that $\sum_{i\in[i_\tau]}a_i^3\le\sigma^2_{i_{\tau}}a_{i_{\tau}}$, so by the
Berry--Esseen theorem (Theorem~\ref{thm:BE}), with $Z$ having the standard normal distribution,
\begin{align*}
\Pr\left[\left|X_{J}+X_{I}-x\right|\le k\sigma_{i_{\tau}}/f\right]&=\Pr\left[\left|Z+(X_{I}-x)/\sigma_{i_\tau}\right|\le k/f\right]+O\left(a_{i_{\tau}}/\sigma_{i_{\tau}}\right)\\
&\le\Pr\left[\left|Z\right|\le k/f\right]+O\left(a_{i_{\tau}}/\sigma_{i_{\tau}}\right)\\
&=O\left(k/f\right).
\end{align*}
Note that by changing $k$ elements in $\x|_{J}$ we can change the value of $X$ by
at most $ka_{i_{\tau}}$, which is not greater than $k\sigma_{i_{\tau}}/f$ by the choice of $\tau$. So, the probability that we can make $X=x$ without modifying $\x|_{I}$ at all
is $O\left(k/f\right)$. By combining all the above bounds, we obtain the desired result.
\end{proof}

\subsection{Proof of Theorem~\ref{thm:lower-bound}}

Finally, we show how to deduce Theorem~\ref{thm:lower-bound} from Lemma~\ref{thm:recursive}.

\begin{proof}[Proof of Theorem~\ref{thm:lower-bound}]
Let $\delta>0$ be a small constant and let $c=3+\delta$. We
prove that $p_{k}\le n^{-c^{-k-1}}$ for $k\le\log_{3+2\delta}\log n$
and sufficiently large $n$, from which the theorem statement will follow. (In this section all asymptotics are uniform over $k\le\log_{3+2\delta}\log n$). We prove our desired bound on $p_k$ by induction on $k$. For $k=0$, as mentioned in the introduction, the Erd\H os--Littlewood--Offord theorem gives
$$p_0=\Theta(n^{-1/2})\leq n^{-c^{-1}}.$$
Next, consider
some $0<k\le\log_{3+2\delta}\log n$ and suppose $p_{k'}\le n^{-c^{-k'-1}}$ for all $k'<k$. Observe that
$$c^{k}\le(3+\delta)^{\log_{3+\delta}\log n/\log_{3+\delta}(3+2\delta)}=\left(\log n\right)^{1-a},$$
for some constant $0<a<1$ depending on $\delta$, and let $f=n^{c^{-k}/3}\ge \exp({(\log n)^a/3})$. For some $n'\ge n-4(k+1)f^2\log n=n-o(n)$, Lemma~\ref{thm:recursive} says that
$$
p_{k} \le\sum_{\ell=1}^{k}\left(4(k+1)f^{2}\log n\right)^{\ell}p_{k-\ell}\left(n'\right)+O\left(k/f+1/n\right).
$$
Observe that $k/f+1/n=o\left( n^{-c^{-k-1}}\right)$ and $\log(n-o(n))=\log n+o(1)$, so it follows that
$$
 p_k\le\sum_{\ell=1}^{k}\left(n^{2c^{-k}/3}\log^{2}n\right)^{\ell}\exp\left(-c^{-k+\ell-1}\left(\log n+o\left(1\right)\right)\right)+o\left( n^{-c^{-k-1}}\right).
$$
Now, recalling that $c^k\le (\log n)^{1-a}$, for $1\le\ell\le k$ we have
\begin{align*}
&\left(n^{2c^{-k}/3}\log^{2}n\right)^{\ell}\exp\left(-c^{-k+\ell-1}\left(\log n+o\left(1\right)\right)\right)\\
&\qquad=\exp\left(c^{-k-1}\log n\left(\frac{2c}{3}\ell+\frac{2c^{k+1}\log\log n}{\log n}\ell-c^{\ell}\left(1+o\left(1/\log n\right)\right)\right)\right)\\
&\qquad=\exp\left(-c^{-k-1}\log n\left(c^{\ell}-\frac{2c}{3}\ell+o\left(1\right)\right)\right)\\
&\qquad=\exp\left(-c^{-k-1}\log n\left(\frac{c}{3}+o(1)\right)\right).
\end{align*}
(We have used the fact that $c^\ell -(2c/3)\ell\ge c-(2c/3)=c/3$ for $\ell \ge 1$ and $c\ge 2$). Consequently,
\begin{align*}
p_{k} & \le k\exp\left(-\left(\frac{c}{3}+o(1)\right) c^{-k-1}\log n\right)+o\left(n^{-c^{-k-1}}\right)\\
 & =o\left( n^{-c^{-k-1}}\right).
\end{align*}
This concludes the proof of the desired bound on $p_k$, and it follows that if $k=\log_{3+2\delta}\log n$ then
$$p_k\leq n^{-c^{-k-1}}=\exp\left(-c^{-1}c^{-k}\log n\right)=o(1).$$
In particular, since $\delta$ is arbitrary it follows that for any $\varepsilon>0$, $\a\in(\mathbb R\setminus\{0\})^n$ and $x\in \mathbb R$, a.a.s.\ $R_x> (1-\varepsilon)\log_3 \log n$.

\end{proof}

\section{A sequence with low typical resilience}\label{sec:upper-bound}

In this section we prove Theorem~\ref{thm:upper-bound} by constructing a sequence $\a$ such that a.a.s.\ $R_{0}=\left(1+o\left(1\right)\right)\log_{3}\log n$.

Let $X=\sum_{i=1}^n a_i\xi_i$ as in the introduction. To construct a sequence $\a$ that results in low typical resilience, we are looking to improve on the idea of Example~\ref{exa:upper-bound}. We start with the ``nicely behaved'' sequence $\a=(1,1,\dots,1)$, and we look to ``plant'' a small subset $B$ in $\a$ which allows us to ``cancel out'' the typical outcomes of $X$. This leads us to consider the following notion.

\subsection{Additive bases}

An order-$h$ \emph{additive basis} of $\left[n\right]$
is a subset $B\subseteq\left[n\right]$ such that for each $x\in\left[n\right]$,
there are distinct $b_{1},\dots,b_{t}\in B$, $t\leq h$, with $x=b_{1}+\dots+b_{t}$. As an easy example, the reader may note that the key part of the sequence in Example~\ref{exa:upper-bound} was the additive basis $\{1,2,2^2\ldots,2^{\ceil{\log n}-1}\}$ of $[n]$, which is of order $\ceil{\log n}$. In order to improve on Example~\ref{exa:upper-bound} and prove Theorem~\ref{thm:upper-bound}, we wish to include a lower-order additive basis in our sequence $\a$.

 The critical issue with this idea is that our additive basis must be part of the sequence $\a$ itself, and therefore it contributes to the behaviour of the typical sum. For example, if we define $\a$ by taking a sequence of $n'$ ``1''s and combining it with a low-order additive basis of $[n']$, then due to the extra ``weight'' of the additive basis, $X$ can take values (much) larger than $n'$, which are not ``covered'' by the additive basis. This issue was circumvented in Example~\ref{exa:upper-bound} because the size of the basis was equal to its order: we were able to control each element in the basis with our $k=\Theta(\log n)$ changes.

In order to minimize the impact of including an additive basis in $\a$, we need an additive basis with small sum of squares. (Recall that the variance of $X$ is $\sum_{i=1}^na_i^2$, and this controls the typical size of $|X|$). Let $v_{h}\left(n\right)$ be the minimum sum of squares of an order-$h$
additive basis of $\left[n\right]$. That is,
$$v_h(n)=\min \left\{\sum_{b\in B}b^2 \mid \text{ }B \text{ is an order-}h \text{ additive basis of }[n]\right\}.$$
In the following lemma we provide an upper bound on $v_h(n)$.
\begin{lemma}
\label{lem:additive-basis}For $h\ge1$ we have
\[
v_{h}\left(n\right)\le 10^h n^{2+2/\left(3^{h}-1\right)}.
\]
\end{lemma}

Our proof of Lemma~\ref{lem:additive-basis} uses an inductive construction closely resembling a construction of Rohrbach~\cite{Roh37}.
\begin{proof}
The proof is by induction on $h$. For the base case, $h=1$, one can take $B=\left\{1,\dots,n\right\}$. Note that indeed we have
$$\sum_{i=1}^n i^2\leq n^3\le 10^1n^{2+2/(3^1-1)}.$$
Next, consider $h>1$ and assume that for all $n$ we have
$$v_{h-1}\left(n\right)\le 10^{h-1}n^{2+2/\left(3^{h-1}-1\right)}.$$
For what follows it will be convenient to use the identity
\begin{equation}
2+2/(3^h-1)=2\times 3^h/(3^h-1)\label{eq:basis-identity}.
\end{equation}
Set $$m=\ceil{n^{2\times3^{h-1}/\left(3^{h}-1\right)}}$$
and consider an order-$\left(h-1\right)$
additive basis $B'$ of $\left[\floor{n/m}\right]$ with sum of squares $v_{h-1}\left(\floor{n/m}\right)$.

Now, let us define $m\cdot B'=\{mb:b\in B'\}$, and note that $B=\left[m\right]\cup (m\cdot B')$ is an order-$h$ additive basis of
$\left[n\right]$. Indeed, for any $x=mq+r$ (with $q\le\floor{n/m}$ and  $1\le r\le m$), there are $b_{1},\dots,b_{t}\in B'$
with $t\leq h-1$ and $b_{1}+\dots+b_{t}=q$. Then, note that each $mb_{i}\in B$,
and $r\in B$, so we can write $x=mb_{1}+\dots+mb_{t}+r$, which is a sum of at most $h$ elements. So, we have
$$v_{h}\left(n\right) \le m^{3}+m^{2}v_{h-1}\left(\floor{n/m}\right).$$
Now, observe that $n\le m^{(3^{h}-1)/(2\times 3^{h-1})}$. Using \eqref{eq:basis-identity},
\begin{align*}\left(\frac{3^{h}-1}{2\times3^{h-1}}-1\right)\left(2+\frac{2}{3^{h-1}-1}\right) & =\left(\frac{3^{h}-1}{2\times3^{h-1}}-1\right)\left(\frac{2\times3^{h-1}}{3^{h-1}-1}\right)\\
 & =\frac{3^{h}-1-2\times3^{h-1}}{3^{h-1}-1}=1
\end{align*}
so the induction hypothesis gives $v_{h-1}\left(\floor{n/m}\right)\le 10^{h-1}m$. Therefore, $v_{h}\left(n\right)\le (10^{h-1}+1)m^3$. Noting that $\ceil x\le 2x$ for $x\ge 1$, and again using \eqref{eq:basis-identity},
\begin{align*}
v_{h}\left(n\right)
 \le8(10^{h-1}+1)n^{2\times3^{h}/\left(3^{h}-1\right)}
\le10^h n^{2+2/\left(3^{h}-1\right)}.
\end{align*}
This completes the proof. \end{proof}


\subsection{Proof of Theorem~\ref{thm:upper-bound}}

Recall that the key idea for our construction is to ``plant" an additive basis of an appropriate order, with low sum-of-squares, in the all-$1$ sequence. Note that for $k=\log_{3}\log n$ we can use Lemma~\ref{lem:additive-basis} to find an order-$k$ additive basis of $[n]$ with sum-of-squares $O(10^k n^2)=n^2\log^{O(1)}n$. A variance bound of $\sigma^2=n^2\log^{O(1)}n$ is enough to prove that a.a.s.\ $|X|\le n\log^{O(1)}n$, but is not quite enough to prove that a.a.s.\ $|X|\le 2n$, which we need for the additive basis of $[n]$ to be effective. We can address this issue by additionally including a very small number of large powers of 2 in our sequence; by modifying the corresponding signs we will be able to make $|X|\le 2n$. A second consideration is the fact that changing a sign increases $X$ if the sign was negative and decreases $X$ if the sign was positive. In order to guarantee that we can a.a.s.\ use our additive basis to adjust $X$ in either direction, we can include many repetitions of the elements of our basis (so that a.a.s.\ there will be a copy of each element with a positive sign and with a negative sign). These basic ideas are enough for a sequence with typical resilience $O(\log \log n)$, but to optimize our construction for the asymptotically lowest possible resilience requires some additional technical details. In particular we include in our sequence two different additive bases of different orders, each with different amounts of repetition.

\begin{proof}[Proof of Theorem~\ref{thm:upper-bound}] Consider small $\varepsilon>0$ and let
\[
h=\ceil{\log_{3-\varepsilon}\log n},\;h'=\ceil{\log_{3-\varepsilon}\log\log n},\;r=\ceil{\log\log^2\log n}.
\]
We will construct a sequence $\a$ such that a.a.s.\ $R_{0}\le h+h'+r$.

Fix an order-$h$ additive basis $B$ of $\left[\ceil{n/\log^2 n}\right]$ with sum
of squares
\[
\sum_{b\in B}b^{2}=O\left(10^h\left(n/\log^2 n\right)^{2+2/\left(3^{h}-1\right)}\right)=o\left(n^{2}/\log n\right)
\]
(note that $10^h=o(\log^3n)$ for small $\varepsilon$), and similarly fix an order-$h'$ additive basis $B'$ of $\left[\ceil{\log^2 n/\log^2\log n}\right]$
with sum of squares $$o\left(\log^{4}n/\log\log n\right).$$
Note that $|B|=o(n/\log n)$ and $|B'|=o(\log^2(n)/\log \log n)$.

Now, define $\a$ by combining:
\begin{itemize}
\item $\ceil{\log n}$ copies of each $b\in B$ (let $I$
be the corresponding set of indices of $\a$);
\item $\ceil{\log \log^2 n}$ copies of $\ceil{n/\log^2 n}b$
for each $b\in B'$ (let $J$ be the corresponding set of indices);
\item the numbers $m,2m,\dots,2^{r-1}m$ for $m=\ceil{n/\log^2 n}\ceil{\log^2 n/\log^2\log n}$ (let $K$ be the corresponding set
of indices);
\item $n-r-\left|B\right|\ceil{\log n}-\left|B'\right|\ceil{\log \log^2 n}$
copies of the number ``1'' (let $L$ be the corresponding set of
indices).
\end{itemize}
Also, if necessary change one of the ``1''s in the final bullet point to a ``2'' to ensure that $\sum_{i=1}^n a_i$ is even. (This guarantees that $X$ is always even).

Now, consider some $b\in B$ and let $I_{b}$ be the set of indices
corresponding to the copies of $b$ in $\a$. Note that
\[
\Pr\left[\x|_{I_{b}}=\left(1,\dots,1\right)\right]=\Pr\left[\x|_{I_{b}}=\left(-1,\dots,-1\right)\right]\le2^{-\log n}=\frac1n=o\left(\frac1{|B|}\right).
\]

So, by the union bound, a.a.s.\ for each $b\in B$ there is at least
one copy of $b$ associated with a negative sign and one associated
with a positive sign. Similarly, a.a.s.\ for each $b\in B'$ there
is a negative and positive copy of $\ceil{n/\log^2 n}b$. In what follows we assume both these properties hold.

Next, note that
\begin{align*}
\sigma_{I}^{2} & =O\left(\log n\right)o\left(n^{2}/\log n\right)=o\left(n^{2}\right),\\
\sigma_{J}^{2} & =O\left(\log\log n\right)O\left(\left(n/\log^2 n\right)^{2}\right)o\left(\log^{4}n/\log\log n\right)=o\left(n^{2}\right),\\
\sigma_{L}^{2} & \le n,
\end{align*}

so $\sigma^2_{I\cup J\cup L}=o(n^{2})$ and by Chebyshev's
inequality, a.a.s.\ $|X_{I\cup J\cup L}|\le2n$. Assuming this, by
modifying $\x|_{K}$ we can make $\left|X\right|\le2m$. Then, there
are $b_{1},\dots,b_{t}\in B'$ with $t\leq h'$ and $\sum_{i=1}^{t}\xi_{i}b_{i}=\ceil{|X/2|/\ceil{n/\log^2 n}}$,
and we can therefore make $\left|X\right|\le2\ceil{n/\log^2 n}$ by
changing a further $t\leq h'$ signs in $\x|_{J}$. Finally, there are $b_{1},\dots, b_{s}\in B$
with $s\leq h$ and $\sum_{i=1}^{h}\xi_{i}b_{i}=\left|X/2\right|$, so we can make
$X=0$ by changing $s\leq h$ signs in $\x|_{I}$. This completes
the proof.\end{proof}

\section{Asymptotics of $p_k(n)$}\label{sec:p1}

In this section we prove Theorem~\ref{thm:pk-individual}. We stress that throughout this section, $k$ is fixed.

\subsection{Upper bounds}\label{sub:upper}
The upper bound $p_{k}\left(n\right)\le n^{-1/\left(2\times3^{k}\right)}\log^{O\left(1\right)}n$
follows immediately from Theorem~\ref{thm:recursive}, using a similar (but much simpler) induction argument to the one used to prove Theorem~\ref{thm:lower-bound}, as follows.

\begin{proof}
For $k=0$, as mentioned in the introduction, the Erd\H os--Littlewood--Offord theorem gives $$p_0=\Theta(n^{-1/2})=\Theta(n^{-1/(2\times3^0)}).$$ For $k>0$, suppose $p_{k'}\le n^{-1/\left(2\times3^{k'}\right)}\log^{O\left(1\right)}n$ for $k'<k$. Let $f=n^{1/\left(2\times3^k\right)}$. Then, using Lemma~\ref{thm:recursive},
\begin{align*}
p_{k} & \le\sum_{\ell=1}^{k}\left(4f^{2}(k+1)\log n\right)^{\ell}p_{k-\ell}\left(n-o\left(n\right)\right)+O\left(k/f+1/n\right)\\
 & \le\sum_{\ell=1}^{k}n^{\ell/3^k}n^{-1/\left(2\times3^{k-\ell}\right)}\log^{O\left(1\right)}n+O(n^{-1/\left(2\times3^{k}\right)})\\
 & \le\sum_{\ell=1}^{k}n^{-(3^\ell-2\ell)/\left(2\times3^{k}\right)}\log^{O\left(1\right)}n+O(n^{-1/\left(2\times3^{k}\right)})\\
 & \le n^{-1/\left(2\times3^{k}\right)}\log^{O\left(1\right)}n.
\end{align*}
This completes the proof.
\end{proof}
For the tight upper bound $p_{1}\left(n\right)=O(n^{-1/6})$
we will use S\'ark\"ozy and Szemer\'edi's theorem (mentioned in the introduction) which asserts that if $\a$ has distinct elements, then $$\Pr\left[X=x\right]=O\left(n^{-3/2}\right).$$
\begin{proof}[Proof of the upper bound on $p_1(n)$]
Fix any $\a,x$. Suppose there are $g$ distinct values in $\a$, so there are at most $2g$ different ways to affect $X$ by flipping a sign. Just as in the proof of Claim~\ref{claim:small-atom}, for any particular choice of index at which to perform a flip, the resulting sequence $\x'$ has the same distribution as $\x$, so the probability that $X(\x')=x$ after the change is $O(n^{-1/2})$ by the Erd\H os--Littlewood--Offord theorem. The union bound over all possible ways to make one flip (or no flips) then gives
\begin{equation}
\Pr\left[R_{x}\le1\right]=O\left(gn^{-1/2}\right).\label{eq:individual-ELO}
\end{equation}
Alternatively, let $a_{i_{1}},\dots,a_{i_{g}}$ give a representative
for each distinct value and let $I=\left\{ i_{1},\dots,i_{g}\right\} $.
Conditioning on $\x|_{\left[n\right]\backslash I}$ and similarly using S\'ark\"ozy and Szemer\'edi's theorem and the union bound,
\begin{equation}
\Pr\left[R_{x}\le1\right]=O\left(g\times g^{-3/2}\right)=O\left(g^{-1/2}\right).\label{eq:individual-SS}
\end{equation}
No matter the value of $g$, one of \eqref{eq:individual-ELO} or \eqref{eq:individual-SS} gives $\Pr\left[R_{x}\le1\right]=O\left(n^{-1/6}\right)$ (if $g\le n^{1/3}$ then use \eqref{eq:individual-ELO}, otherwise use \eqref{eq:individual-SS}). This completes the proof.
\end{proof}

\subsection{Lower bounds}\label{sec:individual-lower}
First we prove the general lower bound $p_{k}\left(n\right)\ge n^{-1/\left(2\times3^{k}\right)}\log^{O\left(1\right)}n$.

\begin{proof}
Let
$$\sigma_{I}=\sqrt{\sum_{i\in I}a_{i}^{2}},\; \rho_{I}=\sum_{i\in I}a_{i}^{3},$$

and define $\sigma=\sigma_{\left[n\right]}$ and $\rho=\rho_{\left[n\right]}$, for use with the Berry--Esseen theorem (Theorem~\ref{thm:BE}). The proof proceeds in a similar way to Theorem~\ref{thm:upper-bound}, as follows. Let $g=(\varepsilon n/\log n)^{1/(2+2/(3^k-1))}=n^{1/2-1/(2\times 3^k)}\log^{O(1)} n$, for some small $\varepsilon>0$ to be determined (where useful for clarity, asymptotic notation will be uniform over $\varepsilon$). Using Lemma~\ref{lem:additive-basis} fix an order-$k$ additive basis $B$ of $\left[g\right]$ with
sum of squares $$\sum_{b\in B}b^2=\Theta\left(g^{2+2/\left(3^{k}-1\right)}\right).$$
Define $\a$ by combining $\ceil{2\log n}$ copies of each
$b\in B$ (let $I$ be the corresponding set of indices in $\a$), and padding the
remaining $n-\left|B\right|\ceil{2\log n}$ entries with
``1''s. As in Section~\ref{sec:upper-bound}, if necessary we can change a ``1'' to a ``2'' to ensure that $\sum_{i=1}^{n}a_{i}$ is even, and we can show that a.a.s.\ for each $b\in B$ there
is at least one copy of $b$ associated with a negative sign and
one associated with a positive sign. Assume this holds.

Now, we have $\sigma_I^2=\Theta\left(g^{2+2/\left(3^{k}-1\right)}\log n\right)$ and $\sigma^2=n-g+\sigma_I^2$, and since each $a_i\le g$, we also have $\rho= n-g+\rho_I\le n-g+g\sigma_I^2$. By the definition of $g$, this means $\sigma_I^2=\Theta(\varepsilon n)$, so $\sigma^2=\Theta(n)$ and $\rho=O(\varepsilon g n)$. By the Berry--Esseen theorem (Theorem~\ref{thm:BE}), for small enough $\varepsilon$ we have
\[
\Pr\left[\left|X/2\right|\le g\right]=\Theta\left(\frac{g}{\sigma}\right)+O\left(\frac{\rho}{\sigma^3}\right)=\Theta\left(\frac{g}{\sqrt n}\right)-O\left(\frac{\varepsilon g}{\sqrt n}\right)=n^{-1/(2\times3^k)}\log^{O(1)} n.
\]
Now, if $|X/2|\le g$ then there are $b_{1},\dots b_{t}\in B$, $t\le k$, with $\sum_{i=1}^{t}b_{i}=|X/2|$,
and we can therefore make $X=0$ by changing $t$ signs
in $\x|_{I}$. This completes the proof.
\end{proof}

Finally, we prove the sharp bound $p_{1}\left(n\right)=\Omega\left(n^{-1/6}\right)$.

\begin{proof}
The construction is similar to the one given above (with $k=1$), but we include only one copy of each element in $B$. Recalling the base case for the induction in the proof of Lemma~\ref{lem:additive-basis}, define $\a$ by $$a_{1}=\dots=a_{n-g}=1,\;a_{n-g+i}=i,$$ where $g=(\varepsilon n)^{1/3}$ for some $\varepsilon>0$ to be determined. (We will be able to choose an appropriate $\varepsilon$ such that $\sum_{i=1}^n a_i$ is even, without having to change a ``1'' to a ``2''). Let $J=[n-g]$ and $I=[n]\setminus J$. By the same arguments as above, we have $\sigma_I^2=\Theta(g^3)=\Theta(\varepsilon n)$, $\sigma^2=\Theta(n)$, and $\rho_I,\rho=O(\varepsilon g n)$, so using the Berry--Esseen theorem in the same way as in the last proof gives
\[
\Pr\left[\left|X/2\right|\le g\right]=\Theta\left(\frac{g}{\sqrt n}\right)-O\left(\frac{\varepsilon g}{\sqrt n}\right)=\Theta\left(n^{-1/6}\right).
\]
Similarly, we can use the estimates $\sigma^2_I,\sigma^2_J=\Theta(n)$, $\rho_I=O(n^{4/3})$ and $\rho_J=O(n)$, and the Berry--Esseen theorem applied to $X_{I}$ and $X_{J}$,
to show that for large $C$ and any $x\in\mathbb{R}$,
\begin{align*}
\Pr\left[\left|X_{I}\right|>C\sqrt{n}\right] & \le1/C,\\
\Pr\left[\left|X_{J}+x\right|\le2g\right] & =O\left(n^{-1/6}\right).
\end{align*}
So,
\begin{align*}
\Pr\left[\left|X/2\right|\le g\mbox{ and }\left|X_{I}\right|>C\sqrt{n}\right] & =\sum_{x:\left|x\right|>C\sqrt{n}}\Pr\left[\left|X_{J}+x\right|\le2g\right]\Pr\left[X_{I}=x\right]\\
 & =O\left(n^{-1/6}\right)\sum_{x:\left|x\right|>C\sqrt{n}}\Pr\left[X_{I}=x\right]\\
 & =O\left(\frac{n^{-1/6}}{C}\right).
\end{align*}
For large enough $C$, we therefore have
\[
\Pr\left[\left|X/2\right|\le g\mbox{ and }\left|X_{I}\right|\le C\sqrt{n}\right]=\Theta\left(n^{-1/6}\right)-O\left(\frac{n^{-1/6}}{C}\right)=\Theta\left(n^{-1/6}\right).
\]
Now, with $N=n-g$, for any $x$ with $N+x$ even and $|x|\le2C\sqrt{n}$ we have
\begin{align*}
\Pr\left[X_{J}=x\right] & =\binom{N}{\left(N+x\right)/2}/2^{N}\\
 & =\frac{\Theta(1)}{\sqrt{N}\left(1+x/N\right)^{\left(N+x\right)/2}\left(1-x/N\right)^{\left(N-x\right)/2}}\\
 & =\frac{\Theta(1)}{\sqrt{N}\left(1-x^2/N^2\right)^{N/2}\left(1+O\left(x/N\right)\right)^{x/2}}\\
 & =\frac{\Theta(1)}{\sqrt{N}\left(1-O\left(1/n\right)\right)^{O\left(n\right)}\left(1+O\left(1/x\right)\right)^{x/2}}\\
 & =\Theta\left(\frac{1}{\sqrt{n}}\right).
\end{align*}
That is to say, the probabilities $\Pr\left[X_{J}=x\right]$ differ from
each other by at most a constant factor.

Let $s\left(a\right)=\sign\left(\xi_{n-g+a}\right)$. Conditioning
on any choice of $\x|_{I}$ such that $\left|X_{I}\left(\x\right)\right|\le C\sqrt{n}$,
we have
\begin{align*}
\Pr\left[\left|X/2\right|\le g\mbox{ and }\sign\left(X\right)=\sign\left(\xi_{n-g+\left|X/2\right|}\right)\right] & =\sum_{a:0\le a\le g}\Pr\left[X_{J}=2s\left(a\right)a-X_{I}\right]\\
 & =\Theta\left(\sum_{a:0\le a\le g}\Pr\left[X_{J}=-2s\left(a\right)a-X_{I}\right]\right)\\
 & =\Theta\left(\Pr\left[\left|X/2\right|\le g\mbox{ and }\sign\left(X\right)\ne\sign\left(\xi_{n-g+\left|X/2\right|}\right)\right]\right).
\end{align*}
So,
\begin{align*}
 & \Pr\left[\left|X/2\right|\le g\mbox{ and }\sign\left(X\right)=\sign\left(\xi_{n-g+\left|X/2\right|}\right)\mbox{ and }\left|X_{I}\left(\x\right)\right|\le C\sqrt{n}\right]\\
 & \qquad=\Theta\left(\Pr\left[\left|X/2\right|\le g\mbox{ and }\left|X_{I}\left(\x\right)\right|\le C\sqrt{n}\right]\right)\\
 & \qquad=\Omega\left(n^{-1/6}\right).
\end{align*}
But if $\left|X/2\right|\le g$ and $\sign\left(X\right)=\sign\left(\xi_{n-g+\left|X/2\right|}\right)$
then we can modify $\xi_{n-g+\left|X/2\right|}$ to make $X=0$. This completes the proof.
\end{proof}

\section{Concluding remarks and open problems}\label{sec:concluding}

In this paper we have investigated the resilience of the anti-concentration
in the Littlewood--Offord problem. We hope the results and ideas in this paper can be applied to other problems, in particular
to the resilience questions for random matrices raised by Vu~\cite{Vu08}. We would like to draw attention to several interesting open questions.
\begin{itemize}
\item It would be interesting if the polylogarithmic error term could be removed from Theorem~\ref{thm:pk-individual}. This problem is analogous to the situation in the Erd\H os--Moser problem, where S\'ark\"ozy
and Szemer\'edi~\cite{SS} removed a polylogarithmic factor in Erd\H{o}s and Moser's original bound. Indeed, it is due to S\'ark\"ozy and Szemer\'edi's theorem that
we could get the right order of magnitude for $p_{1}\left(n\right)$.
\item We showed that for some $\varepsilon\to 0$, for $k\le\left(1-\varepsilon\right)\log_{3}\log n$, a.a.s.\ $R_x>k$ for any $\a,x$,
and for $k\ge\left(1+\varepsilon\right)\log_{3}\log n$ there is $\a$ such that a.a.s.\ $R_0\le k$. It remains open what
the behaviour is when $k$ is very close to $\log_3\log n$. Is there a ``sharp
threshold'' $k=k(n)$ in the sense that $p_k\to 0$ but $p_{k+1}\to 1$ (or $p_{k+a}\to1$ for some fixed $a$)? This would be analogous to the two-point concentration phenomenon for the chromatic number of random graphs~\cite{two-point}. As pointed out to us by Joel Spencer, there is also the possibility that there is some $f=o(\log\log n)$ such that, if $k=\log_3\log n+\lambda f(n)$, then $p_k$ depends nontrivially on $\lambda$. This would be analogous to the behaviour of the connectivity threshold for random graphs; see~\cite{ER}.
\item The constructions used to prove Theorem~\ref{thm:upper-bound} had a very
special ``layered'' structure, and the
proof of the lower bound in Theorem~\ref{thm:lower-bound} seems to indicate
that this type of structure is necessary for the typical resilience
to be small. It would be interesting to formalize this idea in an
inverse theorem of some kind, and we suspect such a theorem would be very useful for the random matrix questions of Vu mentioned in the introduction. An inverse theorem for Theorem~\ref{thm:pk-individual}
would also be interesting: fixing $k$, what can be said about the
structure of $\a$ given $\max_{x}\Pr\left[R_{x}\le k\right]$?
\item We have considered the setting where $X$ is a linear combination
of independent Rademacher random variables. As suggested to us by Van Vu, one can consider more generally
the setting where $X$ is a low-degree polynomial. The anti-concentration
problem in this setting was initated by Costello, Tao and Vu~\cite{CTV06}
in order to study symmetric random matrices, and was further developed
by many authors, most recently by Meka, Nguyen and Vu~\cite{MNV15}.
Resilience problems in this setting appear to be more difficult
than for the ordinary Littlewood--Offord problem, and are likely to
require new ideas.
\end{itemize}

We would also like to highlight an alternative construction of a sequence $\a$ which results in $\Pr[R_0\le k]\ge 99\%$ for $k=(1+o(1))\log \log n$, due to Svante Janson and Joel Spencer. Let $\a$ consist of all ``1"s, except $1000\log(i+1)$ copies of each $\sqrt{n}/i$ for $1\leq i\leq n^{0.2}$, and $10\log n$ copies of each $j$ for $2\leq j\leq n^{0.3}$. (If the sum of all these numbers is odd, change a single ``1'' to a ``2''). We give a sketch proof that this sequence has the claimed property. First observe that
$$\operatorname{Var}(X)=O\left(n+n^{0.9}\log n +\sum_{i=1}^{n^{0.2}}\frac n{i^2}\log (i+1)\right)=O(n),$$
so by Chebyshev's inequality, $|X|\le L\sqrt n$ for some $L$, with probability at least $99.9\%$. Also, observe that with probability at least $99.9\%$ there is a positive and negative sign associated with each distinct value in $\a$. Indeed, the probability that this fails is at most
$$2\left(n^{0.3} n^{-10}+\sum_{i=1}^{n^{0.2}} (i+1)^{-1000}\right)<0.1\%.$$
Now, consider an outcome of $X$ satisfying both of these properties. By the divergence of the harmonic series, there is $B$ such that $\sum_{i=1}^B\sqrt n/B\ge L\sqrt n$; first make at most $B$ flips among the elements $\sqrt n/i$, for $i\le B$, to obtain $|X|<\sqrt n/B$. Then, the key reason we have resilience $(1+o(1))\log \log n$ is that if $2\sqrt n/(i+1)\le |X|<2\sqrt n/i$ then flipping a sign to add or subtract $2\sqrt n/(i+1)$ results in $|X|\le 2\sqrt n/(i(i+1))$. That is to say, if $|X|\approx\sqrt n/i$ then with one flip we can make $|X|\approx\sqrt n/i^2$, so it takes approximately $\log\log n$ flips to go from $|X|\approx \sqrt n/B$ to $|X|\approx n^{0.3}$, after which we can make $X=0$ with a single flip. We suspect that with some optimization this type of construction could lead to an alternative proof of Theorem~\ref{thm:upper-bound}.

\medskip

{\bf Acknowledgements.}
We warmly thank Svante Janson and Joel Spencer for giving us permission to present their alternative construction. We also thank Van Vu and Joel Spencer for many insightful discussions.

\end{document}